\documentclass{amsart}
\usepackage{amsfonts, graphicx, amsmath, amssymb, color, verbatim}
\usepackage{pdfsync}
\usepackage[margin = 1.5in]{geometry}

\newcommand \lra {\longrightarrow}

\newcommand \absv [1]{\left \lvert #1 \right \rvert }
\newcommand \lp {\left(}
\newcommand \rp {\right)}
\newcommand \la {\langle}
\newcommand \ra {\rangle}
\newcommand{\set}[1]{\left\{ #1 \right\} }
\newcommand{\wt}[1]{\widetilde{#1}}
\newcommand{\norm}[2][]{\left \| #2 \right \|_{#1} }

\newcommand\RR{\mathbb{R}}

\newcommand\Id{\operatorname{Id}}

\newcommand\tPhi{\tilde \Phi}

\newtheorem{theorem}{Theorem}
\newtheorem{lemma}[theorem]{Lemma}
\newtheorem{proposition}[theorem]{Proposition}
\newtheorem{corollary}[theorem]{Corollary}
\newtheorem{dynass}[theorem]{Dynamical Assumption}
\theoremstyle{remark}
\newtheorem{remark}[theorem]{Remark}

\numberwithin{equation}{section}
\numberwithin{theorem}{section}

\DeclareMathOperator \supp {supp}
\DeclareMathOperator \Tr {Tr}

\DeclareMathOperator \Vol {Vol}
\DeclareMathOperator \Crit {Crit}

\DeclareMathOperator \Fix {Fix}

\title{Equidistribution of phase shifts in semiclassical potential scattering}
\author{Jesse Gell-Redman}
\address{Department of Mathematics, University of Toronto}
\email{jgell@math.toronto.edu}
\author{Andrew Hassell}
\address{Mathematical Sciences Institute, Australian National University}
\email{Andrew.Hassell@anu.edu.au}
\author{Steve Zelditch}
\address{Department of Mathematics, Northwestern University}
\email{zelditch@math.northwestern.edu}

\begin{document}

\begin{abstract}

Consider a semiclassical Hamiltonian $H :=   h^{2} \Delta + V -
E$ where $\Delta$ is the positive
Laplacian on $\mathbb{R}^{d}$, $V \in C^{\infty}_{0}(\mathbb{R}^{d})$
and $E > 0$ is an energy level.  We prove that under an appropriate
dynamical hypothesis on the Hamilton flow corresponding to $H$, the
eigenvalues of the scattering matrix $S_{h}(V)$ define a measure on
$\mathbb{S}^{1}$ that converges to Lebesgue measure away
from $1 \in \mathbb{S}^{1}$ as $h \to 0$.  
 \end{abstract}

\maketitle

\section{Introduction}\label{sec:intro}

We consider the semiclassical Hamiltonian
\begin{equation}
  \label{eq:operator}
  H := h^{2}\Delta + V - E
\end{equation}
for a potential function $V \colon \mathbb{R}^{d} \lra \mathbb{R}$
which is smooth and compactly supported, and the associated
family of scattering matrices $S_{h}$, defined in
\eqref{eq:scmatrixdef} below.  The goal of this paper is to study the
asymptotic distribution of the eigenvalues $e^{i\beta_{h,n}}$ of
$S_{h}$, the so-called phase shifts, in the limit $h \to 0$.  To this end, we define a measure $\mu_{h}$ on $\mathbb{S}^{1}$ by
\begin{equation}
  \label{eq:integralh}
  \la \mu_{h}, f \ra :=
  \frac{1}{c_{V}}(2\pi h)^{d-1}\sum_{\mbox{spec}(S_{h})} f(e^{i\beta_{h,n}})
\end{equation}
 for a continuous function $f \colon \mathbb{S}^{1}
\lra \mathbb{C}$.
Here $c_{V}$ is a constant related to the classical Hamiltonian flow
of $H$.  Specifically,
\begin{equation}\label{eq:constant}
  c_{V} = E^{(d-1)/2} \Vol (\mathcal{I}),
\end{equation}
where $\Vol (\mathcal{I}) $ is the
volume with respect to Liouville measure of the subset $\mathcal{I}$ of
$T^{*}\mathbb{S}^{d-1}$ of incoming bicharacteristic rays that
interact with the potential.  See Section \ref{sec:trace}, in
particular \eqref{eq:liouville} and \eqref{eq:interaction} for precise definitions.

Our main theorem, which follows immediately from Theorem
\ref{thm:mainthm} below, is the following.
\begin{theorem}\label{thm:prethm}
  Let $f \colon \mathbb{S}^{1} \lra \mathbb{C} $ be a continuous function satisfying
\begin{equation}\label{eq:compactsupport}
1 \not \in \supp f .
\end{equation}  If $V$ is non-trapping at energy $E$ and the sojourn relation
  associated to $H$ satisfies Assumption \ref{thm:dynass} below, then
  \begin{equation}
    \label{eq:equicompact}
    \lim_{h \to 0} \la \mu_{h}, f \ra = \frac{1}{2\pi} \int_{0}^{2\phi}
    f(e^{i\phi}) d\phi,
  \end{equation}
where the pairing on the left is that in \eqref{eq:integralh}.
\end{theorem}

The \textbf{sojourn relation}, described in Section \ref{sec:trace},
is related to the incoming and outgoing data of integral curves
of the classical flow associated to $H$, and it generalizes the concept of scattering angle
\cite{RSIII}.  Assumption \ref{thm:dynass} implies that, but it
stronger than, the statement that the set
of bicharacteristic rays that interact with the potential but pass
through it undeflected has
 measure zero.  The link between $S_{h}$ and the sojourn relation
comes from the fact, proven in \cite{HW2008} with earlier results in
\cite{A2005, Maj1976, RT1989, V,
 G1976} and also \cite{GMB}, that the
quantum scattering map $S_{h}$ is a semiclassical FIO, i.e.\ that the integral
kernel of the scattering matrix is an oscillatory integral whose
canonical relation is the sojourn relation, as we describe in
Section \ref{sec:trace}.  

In Section \ref{sec:weightedequi} 
we deduce the following corollary to Theorem \ref{thm:prethm}, which
says asymptotically how many eigenvalues of $S_{h}$ lie in a closed sector of
$\mathbb{S}^{1}$ not containing $1$.
\begin{corollary}\label{thm:countingcor}
  Given angles $0 < \phi_{0} < \phi_{1} < 2 \pi$, let $N_{h}(\phi_{0},
  \phi_{1})$ denote the number of eigenvalues $e^{i\beta_{h,n}}$ of
  $S_{h}$ with $\phi_{0} \le \beta_{h,n} \le \phi_{1}$ modulo $2\pi$.
  Then
  \begin{equation}
    \label{eq:number}
\lim_{h \to 0} (2\pi h)^{(d - 1)}   N_{h}(\phi_{0},
  \phi_{1}) = c_{V}  \frac{\phi_{1} - \phi_{0}}{2\pi}.
  \end{equation}
\end{corollary}

\begin{remark} Notice the formal resemblance between  \eqref{eq:number} and the standard asymptotic formula for the number of eigenvalues of a semiclassical operator. Namely, suppose $A_h$ is a self-adjoint semiclassical pseudodifferential operator of order $0$ on a compact manifold $M$ of dimension $d-1$, with $\sigma(A_h)(x, \xi) \to \infty$ as $|\xi| \to \infty$. Let $\wt{N}_h(E)$ denote the number of eigenvalues of $A_h$ that are $\leq E$. Then 
 \begin{equation}
    \label{eq:numberpseudo}
\lim_{h \to 0} (2\pi h)^{(d - 1)}   \wt{N}_{h}(E) = \operatorname{vol}\{ (x, \xi) \mid \sigma(A_h)(x, \xi) \leq E \}. 
  \end{equation}
See \cite{DS1999}, \cite{Zw2012}.
\end{remark}

 In \cite{DGHH2013}, Datchev, Humphries, and the first two authors considered the case where the
potential $V$ in \eqref{eq:operator} is central, i.e.\ depends only
on $r = \absv{x}$.  In the central case, the eigenfunctions
of $S_{h}$ are the spherical harmonics, and they showed \cite[Thm
1.1]{DGHH2013}, that under certain assumptions on the scattering angle
function associated to $H$, the eigenvalues of the spherical harmonics with
angular momentum less than $R\sqrt{E}/h$, where $R$ is the
radius of the convex hull of the support of $V$,
equidistribute around the unit circle $\mathbb{S}^{1}$ as $h \to 0$.
See Section \ref{sec:comp} for a comparison of that work's results
with the results established here.  The idea for tackling the case
of non-central potentials using trace formulae comes from previous works of the third author \cite{Z1992},
\cite{Z1997}, where the distribution of eigenvalues of quantum maps was analyzed.

There is a wealth of work on the asymptotic properties of phase
shifts, including notably work from the 80's (e.g.\ papers by various
combinations of Birman, Sobolev, and Yafaev \cite{BY1980},
\cite{BY1982}, \cite{SY1985}),  and also more recent work (e.g.\ that of Doron and
Smilansky \cite{Sm1992}.)  The corresponding inverse problem
-- determining a potential or an obstacle from the scattering matrix
or other scattering data -- has also been pursued, \cite{G1976, Maj1976,
  KK, MT}.
We refer the reader to \cite{DGHH2013} for
yet more literature review.  See also \cite{Nov, N}.

\medskip

\noindent \textbf{Reduction to $E = 1$:}  Recall that the scattering matrix $S_{h}$ can be defined in terms of
generalized eigenfunctions as follows.  For $\phi_{in} \in
C^{\infty}(\mathbb{S}^{d-1})$, there is a unique solution to $H
u = 0$ satisfying
\begin{equation}
  \label{eq:generalizedEfn}
  u = r^{-(d-1)/2} \lp e^{-i \sqrt{E} r / h} \phi_{in}(\omega) + e^{i
    \sqrt{E} r/ h} \phi_{out}(-\omega) \rp + O(r^{-(d + 1)/2}),
\end{equation}
see e.g.\ \cite{GST}.  By definition
\begin{equation}\label{eq:scmatrixdef}
S_{h}(\phi_{in}) := e^{i\pi(d-1)/2}\phi_{out}.
\end{equation}
Below, we will refer to $\phi_{in}$ and $\phi_{out}$ as the
\textbf{incoming} and \textbf{outgoing data} of $u$.

One checks that 
\begin{equation}
  \label{eq:Eisone}
  S_{h, V}(E) = S_{\wt{h}, \wt{V}}(1),
\end{equation}
where $\wt{h} = h / \sqrt{E}, \wt{V} = V / E$, and $S_{h', V'}(E')$
denotes the scattering matrix for $(h')^{2}\Delta + V' - E'$.  Using
\eqref{eq:Eisone}, it is straightforward to conclude all the theorems
above and below from the same theorems in the case $E = 1$.  Thus we
assume the $E = 1$ for the remainder of the paper.


\section{Dynamics}\label{sec:trace}

As we describe in Section \ref{sec:scatteringmatrix}, the integral kernel
of $S_{h}$ is a
Legendrian-Lagrangian distribution associated to the sojourn
relation of $H$, the sojourn relation being a Legendrian submanifold $L \subset T^{*}\mathbb{S}^{d-1} \times
T^{*}\mathbb{S}^{d-1} \times \mathbb{R}$ related to the classical
Hamilton flow of $H$.  Here $T^{*}\mathbb{S}^{d-1} \times
T^{*}\mathbb{S}^{d-1} \times \mathbb{R}$ 
is endowed with the contact form $\pi_{L}^{*}\chi + \pi_{R}^{*}\chi -
d\tau$, $\chi$ being the canonical
one-form on $T^{*}\mathbb{S}^{d-1}$ given by $\zeta \cdot dz$ in local
coordinates $z$ and dual coordinates $\zeta$, and $\pi_{L}$ and $\pi_{R}$
being projections from $T^{*}\mathbb{S}^{d-1} \times
T^{*}\mathbb{S}^{d-1}$ onto the left and right factors, respectively.
See \cite{GS1994} for a review of the relevant symplectic geometry.  The manifold $T^{*}\mathbb{S}^{d-1}$ admits a
natural measure $\mu$, the \textbf{Liouville measure}, equal to the top
exterior power of the canonical
symplectic form $d\chi$.  Precisely,
\begin{equation}
  \label{eq:liouville}
  \mu = \absv{dz d\zeta}.
\end{equation}
To complete the definition of the constant $c_{V}$ in
\eqref{eq:constant}, it remains to define the set $\mathcal{I}$, which
we proceed to do now.

\medskip

\noindent \textbf{Review of classical dynamics:}

First, from \cite{S}.
Consider Newton's equations of motion
\begin{equation}\label{eq:bichar}
\ddot{x}(t) = F(x(t)), \;\; F = - 2 \nabla V.
\end{equation}
Since it adds no complexity at the moment, we relax the assumption on
$V$, assuming as in \cite{S} only that $V(x) = O(\absv{x}^{-2 -
  \epsilon}) = \mbox{Lip}(V)$, where $\mbox{Lip}$ is the Lipschitz
constant of $V$.  Given a solution to \eqref{eq:bichar}, the quantity $E :=
\absv{\dot  x(t)}^{2} + V(x(t))$ is a constant of the motion.  
We seek solutions of the form
\begin{equation}\label{eq:bicharform}
x(t) = a + t b + u(t), \;\;\ \lim_{t \to - \infty} |u(t)|  + |\dot{u}(t)| \to 0. 
\end{equation}
Write the equation for $u$ as
\begin{equation}\label{eq:u}
u(t) = \int_{- \infty}^t ds \int_{- \infty}^s F(a + b \tau + u(\tau)) d \tau ds.
\end{equation}
Lemma 1 of \cite{S} shows that for this $u$, $x(t)$ defined as in
\eqref{eq:u} satisfies Newton's law and \eqref{eq:bicharform}. We will assume that
$V$ is \textbf{non-trapping} at energy $1$,
meaning that every solution $x(t)$ to \eqref{eq:bichar} with $E = 1$ goes
to infinity both as $t \to - \infty$ and $t \to +\infty$.  One
checks that $x(t)$ also has the 
form in \eqref{eq:bicharform} as $t \to \infty$ with
$$u(t) = \int_t^{\infty} \int_{s}^{\infty} F(a + b tau + u(\tau)) d
\tau ds. $$  Thus we have the following.
\begin{theorem}
For all $a, b
  \in \mathbb{R}^{n}$ with $b \neq 0$, there exists a unique solution
  $x_{a,b}(t)$ satisfying Newton's law such that
$$\lim_{t \to - \infty} |x(t) - a - b t| + |\dot{x}(t) - b|  = 0. $$
If in addition $V$ is non-trapping at energy $E$ and $\absv{b} = \sqrt{E}$,
then there exist $c,d \in \mathbb{R}^{n}$ with $\absv{d} = \sqrt{E}$ such
that 
$$\lim_{t \to \infty} |x(t) - c - d t| + |\dot{x}(t) - d|  = 0. $$
\end{theorem}

In \eqref{eq:sojourn} below, we give a concrete definition of the
sojourn map, but for the moment we discuss the dynamics in the general
setting of \cite{BP}.  
The sojourn map ${\mathcal S}: T^* \mathbb{S}^{d-1} \to T^*
\mathbb{S}^{d-1}$ is a symplectic reduction of the classical
scattering map of \cite{H,S,BP, NT}.  Simon \cite{S}
denotes the scattering map by $(\Omega^-)^{-1} \Omega^+$,
where $\Omega^{\pm}$ are classical wave operators, defined by
\begin{equation}
  \label{eq:2}
  W^{\pm} = \lim_{t \to \pm \infty} \Phi_{-t} \circ \Phi^{0}_{t},
\end{equation}
where $\Phi$ is the solution operator at energy $E$ for Newton's
equations of motion \eqref{eq:bichar} and $\Phi^{0}$ is the solution
operator for Newton's law with $V \equiv 0$, i.e. it is standard
geodesic flow on $\mathbb{R}^{d}$.
In \cite{BP} the wave operators are denoted $W^{\pm}$ and
the fixed energy scattering map is denoted $S_E := W_{+}^{-1} \circ W_{-}$. It is defined
on the energy surface $\Sigma_E^0 = \{(x,\xi) \in
  T^{*}\mathbb{R}^{d} : \absv{\xi} = \sqrt{E}\}$ for
the free Hamiltonian and then on its reduction $\tilde{\Sigma}_E^0$,
i.e.\ on the set of free orbits of $\Phi^{0}$ of energy $E$. This
quotient symplectic manifold can be identified with a transversal
$\Gamma \subset \Sigma^0_E$. 
As is pointed out in \cite{BP}, the reduced symplectic form  $\omega_E$
 on $\tilde{\Sigma}_E^0$ is exact
i.e. $\omega_E = d \alpha_E$ for some $1-$form $\alpha_{E}$. Denote
the classical scattering map at
energy $E$ by $S_E$. Define a function $\tau_E$ (up to addition of a constant)
by
$$S_E^* \alpha_E - \alpha_E = d \tau_E. $$
In (6.4) of \cite{BP} the function $\tau_E$
is denoted $\Delta_0 = \Delta \circ W_-$ where
\begin{equation}\label{eq:almostsojourn}
\Delta(x) = - \int_{-\infty}^{\infty} \iota_X \alpha_{E} \circ \Phi_t (x) dt, \;\; x \in \Sigma_E.
\end{equation}
Here, $X = \dot \Phi_{t}$ is the Hamilton vector field of the Hamiltonian, which in
our context is $\absv{\xi}^{2} + V(x)$, and the
action form $\alpha_{E}$ is chosen so that $\iota_{X_0} \alpha_{E} =
0$ on $\Sigma_E^0$.

\medskip

\noindent \textbf{Explicit sojourn relation and the interaction
  region:}
We now resume our standing assumption that $E = 1$ and give a concrete
definition of the sojourn map $\mathcal{S} := \mathcal{S}_{1}$,
\begin{equation}
  \label{eq:sojourn}
  \mathcal{S} \colon T^{*}\mathbb{S}^{d-1} \lra T^{*}\mathbb{S}^{d-1}.
\end{equation}
Identify $\mathbb{S}^{d-1} \subset
\mathbb{R}^{d}$ with the unit sphere, and given $\omega \in
\mathbb{S}^{d-1}$, identify $T_{\omega}^{*}\mathbb{S}^{d-1}$ with
$\omega^{\perp}$, the space of vectors in $\mathbb{R}^{d}$ orthogonal to
$\omega$.  Given $\eta \in \omega^{\perp}$, there is a unique
bicharacteristic ray, i.e.\ a unique solution $(x_{\omega,\eta},
\xi_{\omega,\eta})$ to \eqref{eq:bichar}, satisfying
\begin{equation}
x_{\omega, \eta}(t) = t  \omega
+ \eta  \mbox{  for   } t << 0.
\end{equation}
By the non-trapping assumption, for $t >>
0$, $\ddot x_{\omega, \eta} = 0$, so the following definition makes sense,
\begin{equation}
  \label{eq:sojourndef}
  \begin{split}
    \mathcal{S}(\omega,\eta) := (\omega', \eta') \mbox{ where }
    x_{\omega,\eta} = t  \omega' + \eta'  \mbox{ for
    } t >>0.
  \end{split}
\end{equation}
We now interpret the \textbf{sojourn time} $\tau(\omega, \eta) :=
\tau_{E = 1}(\omega, \eta) $ as in \eqref{eq:almostsojourn} in our context, using the notation from
the previous subsection.
For Hamiltonian systems of the form \eqref{eq:bichar} on $\mathbb{R}^{d}$ the
action $\alpha := \alpha_{E = 1}$ restricted to $\Sigma_{E = 1}^{0} = \{ (x, \xi) \in
  T^{*}\mathbb{R}^{d} : \absv{\xi}^{2} = 1\}$ is given by
$\chi_{0} + dF$ where $\chi_{0} = \sum_{i = 1}^{d}\xi^{i} dx^{i}$ is
the canonical $1-$form on $\mathbb{R}^{2d}$ and, writing
$\Sigma_{E = 1}^{0} = \set{ (\omega t + \eta, \omega) : (\omega, \eta)
  \in T^{*}\mathbb{S}^{d-1}}$, we have $F(\omega t + \eta, \omega) = - t$.
(See \cite[Sect.\ 5]{BP} for details.) 
It follows that
\begin{equation}
  \label{eq:time}
  \begin{split}
    \tau(\omega, \eta) &= \int_{- \infty}^{\infty} \iota(X) \alpha_{E} \circ (x_{\omega,
      \eta}(s), \dot x_{\omega, \eta}(s)) ds \\
    &= \lim_{t \to \infty}  \int_{-t}^{t} \absv{\xi}^{2}((x_{\omega,
      \eta}(s), \dot x_{\omega, \eta}(s)) +  (XF) \circ (x_{\omega,
      \eta}(s), \dot x_{\omega, \eta}(s)) ) ds \\
    &= \lim_{t \to \infty}  \int_{-t}^{t} (1 - V)(x_{\omega,
      \eta}(s), \dot x_{\omega, \eta}(s)) ds +  F(t) - F(-t) \\
    &= \lim_{t \to \infty}  \int_{-t}^{t} (1 - V)(x_{\omega,
      \eta}(s), \dot x_{\omega, \eta}(s)) ds +  2t. \\
  \end{split}
\end{equation}
The total sojourn relation $L$, defined by
\begin{equation}
  \label{eq:5}
  L := \set{(\omega, \eta, \omega', - \eta', \tau(\omega, \eta)) : (\omega, \eta) \in
    T^{*}\mathbb{S}^{d-1} \mbox{ and } \mathcal{S}(\omega, \eta) = (\omega', \eta')},
\end{equation}
is a Legendrian submanifold of $T^{*}\mathbb{S}^{d-1} \times
T^{*}\mathbb{S}^{d-1} \times \mathbb{R}$ with the contact form
described above.  In particular, $\mathcal{S}$ is a symplectomorphism.

We define
the \textbf{interaction region}
\begin{equation}
  \label{eq:interaction}
  \mathcal{I} := \set{ (\omega, \eta) : \mbox{Image}(x_{\omega,\eta}) \cap \supp V
    \neq \varnothing}.
\end{equation}
Note that, if $(\omega, \eta) \in \mathcal{I}$, then $\mathcal{S}(\omega, \eta)
= (\omega', \eta') \in \mathcal{I}$, since $x_{\omega,
  \eta} = t \omega' + \eta'$ for $t >> 0$, hence the straight
line $t  \omega' + \eta'$ intersects $\supp V$, and thus $x_{\omega',
  \eta'} \cap \supp V \neq \varnothing$.  Similarly,
$\mathcal{S}(\mathcal{I}^{c}) \subset \mathcal{I}^{c}$, so since
$\mathcal{S}$ is invertible by uniqueness of solutions to ODEs,
\begin{equation*}
    \mathcal{S}(\mathcal{I}) = \mathcal{I} \quad \mbox{ and } \quad
  \mathcal{S}(\mathcal{I}^{c}) = \mathcal{I}^{c}.
\end{equation*}

Finally, we define the $l^{th}$ interacting fixed point set for $l \in \mathbb{Z}$,
\begin{equation}  
\label{eq:interactionfixed}
  \mathcal{F}_{l} := \set{ (\omega, \eta) \in \mathcal{I} : \mathcal{S}^{l}(\omega,\eta)
    = (\omega,\eta) }.
\end{equation}
We will make the following assumption.
\begin{dynass}\label{thm:dynass}
The sets $\mathcal{F}_{l}
  \subset T^{*}\mathbb{S}^{d-1}$ in \eqref{eq:interactionfixed} satisfy
  \begin{equation}
    \label{eq:zerofixed}
    \Vol(\mathcal{F}_{l}) = 0 \mbox{ for all } l \in \mathbb{Z}.
  \end{equation}
\end{dynass}

\begin{remark}
  The existence of potentials satisfying Assumption \ref{thm:dynass}
  was established in \cite{DGHH2013}.  Indeed, Assumption
  \ref{thm:dynass} is weaker than the dynamical assumption made in
  \cite{DGHH2013}, and example of potentials satisfying the strong
  assumption of that paper are established therein.  See Section
  \ref{sec:comp} for details.

The authors conjecture that Assumption \ref{thm:dynass} holds for
generic potentials $V$, though we do not pursue this question here.
There is a wealth of research on the topic of generic geodesic and
Hamiltonian flows, going back to Klingenberg-Takens \cite{KT} in the
case of Riemannian metrics.  See also \cite{PS} and \cite{PS2} for
results in the setting of obstacle scattering.  
\end{remark}

\begin{remark}\label{thm:scalingremark}
The total sojourn relation defined here agrees with the definition in 
\cite[Section 15]{HW2008}. To see this, write the
Hamiltonian system in \eqref{eq:bichar} in polar coordinates $(r, \omega)$.  In
particular, the Lagrangian becomes $h = (\rho^{2} + \theta^{2}/r^{2} +
V)/2$ where $\rho$ is dual to $r$ and $\theta$ is dual to $\omega$.  In
the notation of the paragraph preceding Lemma 15.3 in \cite{HW2008}, one checks that
for a bicharacteristic of the form $t  \omega + \eta'$ where
$\eta' \perp \omega$, one has
$\mu = \eta' / t + O(1/t^{2})$ as $t \to \infty$. This yields $M = \eta'$ (NB:  the definition
of $M$ in \cite{HW2008} has a typo; it should be $M := \lim_{t \to \infty}
\mu/x$).  From this we see that \eqref{eq:5} agrees with the definition of the total sojourn relation on \cite[p680]{HW2008}. 
\end{remark}

\begin{figure}
  \centering
  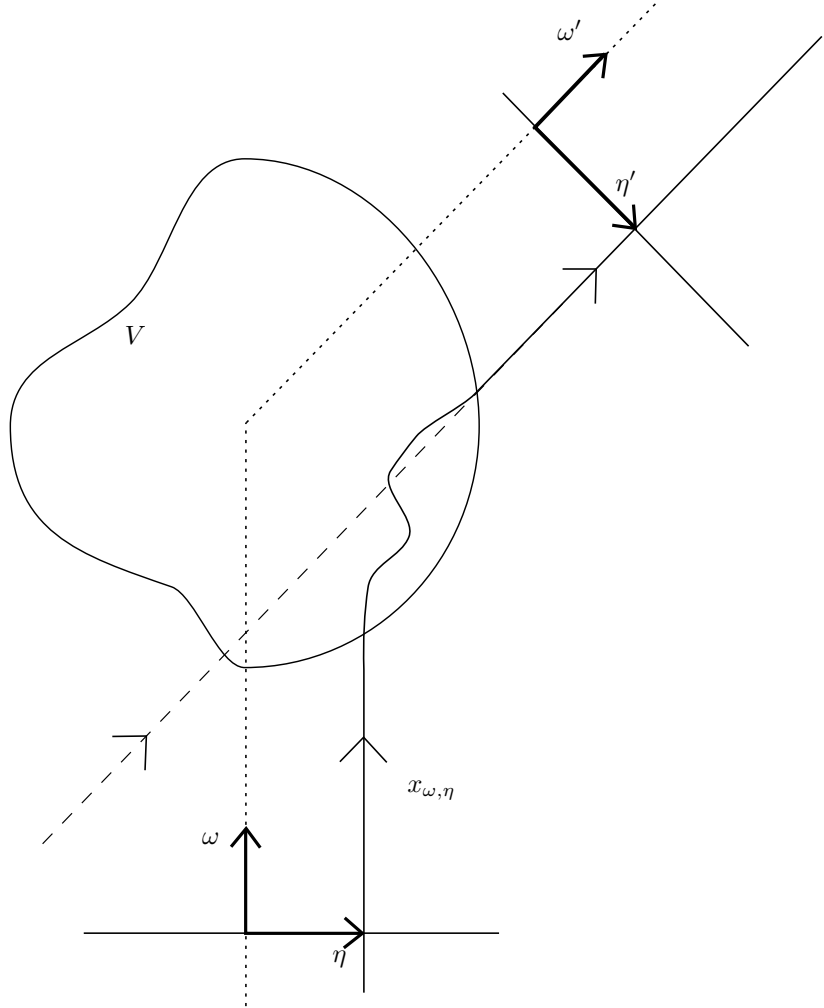
  \caption{The scattering relation.  Here $(\omega, \eta)$ lies in the
    interaction region $\mathcal{I}$.  The long-dashed line depicts how
  the outgoing data $(\omega', \eta')$ is also in $\mathcal{I}$.}
  \label{fig:widetildeL}
\end{figure}


\section{Semiclassical scattering matrix}\label{sec:scatteringmatrix}
In this section we collect some information about the semiclassical scattering matrix needed in the proof of Theorem~\ref{thm:prethm}. 

 As described in \cite{DGHH2013} and proven in \cite{A2005} and
  \cite{HW2008},  the integral kernel of $S_{h}$ can be decomposed as
 \begin{equation}\label{eq:decomposition}
S_{h}
  = K_{1} + K_{2} + K_{3}
\end{equation}
with the $K_{i}$ as follows.

  First, $K_{1}$ is a semiclassical Fourier Integral Operator with
  compact microsupport.  This means that 
  its kernel is a Lagrangian distribution equal to a finite sum of terms of
  the form
  \begin{equation}
    \label{eq:scfiopart}
    h^{-(d-1)/2 -N/2}\int_{\mathbb{R}^{N}} e^{i \Phi(\omega, \omega',
      v)/h} a(\omega,\omega', v, h) dv,
  \end{equation}
where $a$ is a smooth, compactly supported function on
$\mathbb{S}^{d-1} \times\mathbb{S}^{d-1} \times \mathbb{R}^{N}_{v} \times
[0, h_{0})_{h}$, and $\Phi$ is a smooth phase function which parametrizes the
sojourn relation $L$ locally.  We describe briefly what it means to
parametrize $L$ locally; on the critical locus
\begin{equation}
  \label{eq:criticalset}
  \mbox{Crit}(\Phi) = \set{ (\omega, \omega',
      v) ; D_{v}\Phi(\omega, \omega',
      v) = 0}
\end{equation}
the Hessian $D_{\omega, \omega', v}\Phi$ has full
rank and the map
\begin{equation}
  \label{eq:parametrization}
  \begin{split}
    \mbox{Crit}(\Phi) &\lra T^{*}\mathbb{S}^{d-1} \times
    T^{*}\mathbb{S}^{d-1} \times \mathbb{R} \\
    (\omega, \omega', v) &\longmapsto (\omega, D_{\omega}(\Phi),
    \omega', D_{\omega'}(\Phi), \Phi),
  \end{split}
\end{equation}
restricted to the complement of the set $\set{ (\omega, \omega',
  v, h) : a(\omega, \omega',
  v, h) = O(h^{\infty})}$ is a diffeomorphism between $\Crit(\Phi)$
and an open subset of $L$.  (Here and below $O(h^{\infty})$ denotes a quantity that is bounded by
$C_{N}h^{N}$ for each $N > 0$ and $h$ sufficiently small.)
The microsupport condition on $K_{1}$ means that the amplitudes $a$ in
\eqref{eq:scfiopart} satisfy
  \begin{equation}
    \label{eq:10}
    a(\omega, \omega', v, h) = O(h^{\infty}) \mbox{ on those }
    (\omega, \omega', v) \mbox{ mapped via \eqref{eq:parametrization}
      into } \mathcal{I}_{2\epsilon}^{c}.
  \end{equation}
In particular,  the canonical
relation of $K_{1}$ is given by the projected sojourn relation
\begin{equation*}
  C := \set{(\omega, \eta, \omega', -\eta') : \mathcal{S}((\omega,
    \eta) = (\omega', \eta')},
\end{equation*}
This is just the projection of the sojourn relation $L$ off the $\mathbb{R}$
factor.  (For more details about semiclassical FIOs see e.g.\ \cite{Zw2012}, \cite{D1974}.)

The $K_{2}$ term is a semiclassical pseudodifferential operator with
microsupport disjoint from $\mathcal{I}$, say outside 
\begin{equation}
\{ |\eta| \leq R_* \} \mbox{ where } \supp V \subset B_{R^{*}},
\label{Rstar}\end{equation}
$B_{R^{*}}$ being the ball of radius $R^{*}$ centered at the origin.

(Though it will not be used directly, we mention that $K_{2}$ is
microlocally equal to the identity outside a compact set in phase space.
Specifically, it
is a sum of terms of the form
\begin{equation}
  \label{eq:pseudopart}
  (2\pi h)^{-(d-1)} \int e^{i(z - z')\cdot \zeta / h} b(z, 
  \zeta, h) d\zeta
\end{equation}
in local coordinates, where
  \begin{equation}
    \label{eq:6}
    b(z, \zeta, h) = 1 + O(h^{\infty}) \mbox{ for } \absv{\zeta} >
    R^{*} + \delta_{0},
  \end{equation}
for any $\delta_{0}$.  Instead of using \eqref{eq:6} we will use the
exponential bounds in Lemma \ref{thm:veryclosetozero} below.)
 
Finally, the $K_{3}$ term is a smooth function on
$\mathbb{S}^{d-1} \times \mathbb{S}^{d-1} \times [0,
h_{0})_{h}$ satisfying
\begin{equation}\label{eq:K3}
K_{3} = O(h^{\infty}).
\end{equation}

Moreover, by \cite[Lemma 3.1]{DGHH2013} the Maslov line
bundle of the canonical relation $C$ is canonically trivial, and with respect to this
trivialization, the principal symbol of $S_{h}$
is given in terms of the canonical half-density $\absv{d\omega
  d\eta}^{1/2}$ on $\mathcal{M}$ by
\begin{equation}
  \label{eq:principalsymbol}
  \sigma(S_{h}) =  \absv{d\omega
  d\eta}^{1/2}.
\end{equation}
In particular, 
\begin{equation}
  \label{eq:principalsymbolisId}
  \sigma(S_{h} - \Id) = 0 \mbox{ on } T^{*}\mathbb{S}^{d-1} - \mathcal{I}.
\end{equation}


\section{Trace formula}

Theorem \ref{thm:mainthm} below (and thus Theorem \ref{thm:prethm})
will be proven using the following trace formula.

\begin{proposition}\label{thm:traceone}
  Assume that $V$ is non-trapping at energy $1$ and that $H$ in
  \eqref{eq:operator} satisfies Assumption \ref{thm:dynass}. For each polynomial
  $p$ on $\mathbb{C}$
satisfying
  \begin{equation}\label{eq:zero}
p(1) = 0,
\end{equation}
we have
  \begin{equation}
    \label{eq:traceformula}
    \Tr p(S_{h}) = \frac{\Vol(\mathcal{I})}{(2 \pi h)^{d-1}} \frac{1}{2 \pi}
    \oint_{\mathbb{S}^{1}} p(e^{i\phi}) d\phi  + o(h^{-(d-1)}),
  \end{equation}
where $\mathcal{I} \subset T^{*}\mathbb{S}^{d-1}$ is the interaction
region \eqref{eq:interaction}.
\end{proposition}

In preparation for the proof of this proposition, we prove an estimate on the operator $S_h - \Id$. 

\begin{lemma}\label{thm:veryclosetozero}
  Let $R > 0$ be such that
  \begin{equation}\label{eq:bigball}
    \supp V \subset B_{R}, 
  \end{equation}
  where $B_{R}$ is the open ball $\set{ \absv{x} < R}$, and choose $R' > R$.  There exist 
  constants $c, C > 0$ depending on $R'$ and $V$ so that for each spherical harmonic $\phi_{l}$
  satisfying $(\Delta_{\mathbb{S}^{d-1}}  - l(l + n - 2))\phi_{l} = 0$
  with $l h \geq R'$, and for $h$ small enough, 
we have
\begin{equation}
  \label{eq:almosteigenfunction}
  \norm[L^{2}]{(S_{h} - 1)\phi_{l}} \le C e^{-cl} \norm[L^{2}]{\phi_{l}}.
\end{equation}
\end{lemma}

\begin{proof}
  Consider the function
  \begin{equation}\label{eq:almostgeneralized}
u_{l, h} := J_{l + (d-2)/2}(r/h)\phi_{l}
\end{equation}
on
  $\mathbb{R}^{d}$ where $\phi_{l}$ is a spherical harmonic with
  angular momentum $l$ and $J_{\sigma}(\zeta)$ is the standard Bessel
  function of order $\sigma$ defined in \cite[Chapter 9]{AS1964}.
  Then $u_{l, h}$ is in the kernel of $h^{2}\Delta -
  1$ and has incoming data equal to $\phi_{l}$ in the sense of
  \eqref{eq:generalizedEfn}. Therefore, $f_{l,h}$ defined by $\lp h^{2}\Delta + V -  1\rp u_{l, h}$ is given by 
  \begin{equation}
    \label{eq:approxgeneralized}
    f_{l,h} = Vu_{l, h}  = V J_{l + (d-2)/2}(r/h)\phi_{l}. 
  \end{equation}
Writing
\begin{equation}\label{eq:changeofvars}
\nu := l +
(d-2)/2 \quad \mbox{ and } \quad  \gamma:= r/(h\nu),
\end{equation}
we can write 
$$
 f_{l,h} = V J_\nu(\nu \gamma) \phi_{l},
 $$
 where $f_{l,h} = 0$ unless $\gamma \leq R/R'$, since we must have $r \leq R$ due to the $V$ factor and $\nu \geq l \geq R'/h$ by assumption. By \cite[9.3.7]{AS1964}, we have 
$$
J_\nu(\nu \gamma) \leq C e^{-\nu(\alpha - \tanh \alpha)}{\nu \tanh \alpha}, \quad \gamma = \operatorname{sech} \alpha.
$$
Therefore, since $\alpha - \tanh \alpha \geq c > 0$ when $\gamma \leq R/R'$, we have 
\begin{equation}\label{eq:greatbound}
\| f_{l,h} \|_{L^2} \leq C e^{-cl}.
\end{equation}

Now consider the outgoing resolvent $R_{h,V}(1 + i0) := (h^{2}\Delta + V - (1 +
i0))^{-1}$, which satisfies $H (h^{2}\Delta
+ V - (1 + i0))^{-1} v = v$ for $v \in
C^{\infty}_{comp}(\mathbb{R}^{d})$.  The main properties of
$R_{h,V}(1 + i0)$ for fixed $h$ can be found in \cite{M1994}.  In particular, for $v \in
C^{\infty}_{comp}(\mathbb{R}^{d})$ and $h$ fixed, 
\begin{equation}
  \label{eq:outgoingresolvent}
  R_{h,V}(1 + i0) v = r^{-(d-1)/2} e^{ir/h} \psi + O(r^{-(d+1)/2})
\end{equation}
 for some $\psi \in C^{\infty}(\mathbb{S}^{d-1})$.  We refer to $\psi$
as the \textbf{outgoing data} of $R_{h,V}(1 + i0) v$.
The function
\begin{equation}
  \label{eq:generalizedforV}
  u_{l, h} - R_{h,V}(1 + i0)(h^{2} \Delta + V - 1)u_{l, h}
\end{equation}
is the unique generalized eigenfunction of  $h^{2}\Delta + V$ of energy $1$ with  incoming data
$\phi_{l}$ in the sense of \eqref{eq:generalizedEfn}.  Thus by the
definition of the scattering matrix in \eqref{eq:scmatrixdef}, $S_{h}(\phi_{l})$ is equal to $\phi_{l}$
minus the outgoing data of $R_{h,V}(1 + i0)(h^{2} \Delta + V - 1)u_{l,
  h}$, up to composition with unitary maps.  Precisely, if $\psi_{l,h}$ is the outgoing data of
$R_{h,V}(1 + i0)(h^{2} \Delta + V - 1)u_{l, h}$ in the sense of
\eqref{eq:outgoingresolvent}, then 
\begin{equation}
  \label{eq:scatteringonsphericalharm}
  S_{h}\phi_{l}(\omega) = \phi_{l}(\omega) - e^{i\pi(d-1)/2}\psi_{l,h}(-\omega)
\end{equation}
The outgoing data $\psi_{l,h}$ is computed using the adjoint of the Poisson operator
$P_{h, V}$ for
$h^{2}\Delta + V - 1$.  It is a straightforward exercise to show that
the adjoint operator $P^{*}_{h, V}$ satisfies
$P^{*}_{h,V}f_{l,h} = (-2ih) e^{-i\pi(d-1)/2} S_{h}^{*}
\mathcal{A}^{*}\psi_{l,h}$, where $\mathcal{A}^{*}$ is pullback by the
antipodal map of $\mathbb{S}^{d-1}$.  Thus
\begin{equation}
  \label{eq:poisson}
\| (S_h - \Id) \phi_l \|_{L^2} = \frac 1h \| P_{h, V}^* f_{l,h} \|_{L^2}. 
\end{equation}
Given a cutoff function $\chi \colon \mathbb{R}^{d} \lra \mathbb{R}$
satisfying $\chi(x) \equiv 1$ for $\absv{x} \le R$ and $\chi(x) \equiv
0$ for $\absv{x} > 2R$, $P^{*}_{h,V}f_{l,h} = P^{*}_{h,V} \chi
f_{l,h}$, and we claim that
\begin{equation}
  \label{eq:finalestimate}
\norm[L^{2}(\mathbb{R}^{d}) \to L^{2}(\mathbb{R}^{d})]{\chi P_{h,V}
  P^{*}_{h,V} \chi}  = \norm[L^{2}(\mathbb{R}^{d}) \to
L^{2}(\mathbb{S}^{d-1})]{P^{*}_{h,V} \chi}^{2} \le C. 
\end{equation}
To prove \eqref{eq:finalestimate}, we use the identity $R_{h, V}(1 +
i0) - R_{h, V}(1 -
i0) = (i/2h) P_{h,V} P^{*}_{h,V}$ (see e.g.\ \cite[Lemma
5.1]{HV1999}).   The estimate in
\eqref{eq:finalestimate} follows from the estimate on
$\chi R_{h, V}(1 \pm
i0) \chi$ given in the main theorem of \cite{VZ2000}, where the estimates
are in terms of weighted Sobolev spaces, the weights of which are
irrelevant thanks to the cutoff $\chi$.

Thus by \eqref{eq:greatbound}, $\norm[L^{2}]{(S_{h} - \Id) \phi_{l}}
\le  C e^{-cl} \norm[L^{2}]{\phi_{l}}$ for some $c > 0$ in the region
$l \geq R'/h$ and the lemma follows.

\end{proof}


\begin{proof}[Proof of Proposition \ref{thm:traceone}]
  
  For the sake of clarity, we first prove the trace formula \eqref{eq:traceformula} for
\begin{equation}\label{eq:simplepoly}
p(z) = z - 1,
\end{equation}
i.e.\
we analyze $S_{h} - \Id,$ which is indeed a trace class operator \cite{Yafaev}. 

We choose a pseudodifferential operator $A_h \in \Psi^{0,
  \infty}(S^{n-1})$ that is microlocally equal to the identity on a
neighbourhood of $\mathcal{I}$, say the set $\{ |\eta| \leq R \}$, and
is microsupported inside $\{ |\eta| \leq R_* \}$, where $R_*$ is as in
\eqref{Rstar}. To be precise, we choose $A = \rho(h^2
\Delta_{S^{n-1}})$, where $\rho(t) = 1$ for $t \leq \sqrt R$ and $0$
for $t \geq \sqrt{R_*}$. 

Then $\Tr (S_h - \Id) = \Tr A_h (S_h - \Id) + \Tr (\Id - A_h)(S_h - \Id)$. We analyze the two parts separately. 

First consider $\Tr (\Id - A_h)(S_h - \Id)$. This can be calculated using the orthonormal basis of spherical harmonics on $S^{n-1}$. With $A_h = \rho(h^2 \Delta_{S^{n-1}})$ as above, we have 
$$
\Tr (\Id - A_h)(S_h - \Id) = \sum_{lm} \langle \phi_{lm},  (\Id -
A_h)(S_h - \Id) \phi_{lm} \rangle = \sum_{lm} \langle (\Id - A_h)
\phi_{lm} , (S_h - \Id) \phi_{lm} \rangle. 
$$
Here $\phi_{lm}$ satisfies $\Delta_{\mathbb{S}^{d-1}} \phi_{lm} = l(l+d-2)
\phi_{lm}$ and for fixed $l$, the index $m$ ranges from $1$ to $d_l$,
where $d_{l}$ is the multiplicity
\begin{equation}\label{eq:multiplicity}
d_l = \dim \ker (\Delta_{\mathbb{S}^{d-1}} - l(l+d-2)) = O(l^{d - 2}).
\end{equation}
The latter bound can be found for example in \cite{taylor:vol2}.
Then $(\Id - A_h) \phi_{lm} = 0$ unless $l \geq R/h$. Moreover, we see
from Lemma~\ref{thm:veryclosetozero} that $\| (S_h - \Id) \phi_{lm} \|
\leq C e^{-cl}$ when $l \ge R/h$. It follows that we get an estimate 
\begin{equation}\label{eq:nearinfinity}
  \begin{split}
    \Big| \Tr (\Id - A_h)(S_h - \Id) \Big| &\leq C \sum_{l \geq R/h}
    d_l e^{-cl} \le C' \sum_{l \geq R/h} e^{-c'l} \le C'' e^{Rc'/h} (= O(h^{\infty})).
  \end{split}
\end{equation}

We now turn to computing $\Tr A_h (S_h - \Id)$. We decompose $S_h =
K_1 + K_2 + K_3$ as in Section \ref{sec:scatteringmatrix}. Recall the
formula for the trace in terms given by integrating the kernel over
the diagonal \cite[Appendix C]{Zw2012}.  Consider the operator $A_h K_3$. We have already noted that $K_3$ is a smooth
function that is $O(h^\infty)$; on the other hand, $A_h$ is a semiclassical pseudo of order $0$ and compact microsupport, so its kernel is $O(h^{-(d-1)})$. Composing, we see that the kernel of $A_h K_3$ is $O(h^\infty)$, so its trace is also $O(h^\infty)$.

 Since $A_h$ and $K_2$ are microsupported on
disjoint sets, the product $A_h K_2$ has Schwartz kernel smooth and
$O(h^\infty)$, so as for $K_{3}$ we conclude that $\Tr(A_{h}K_{2}) = O(h^\infty)$. 

It remains to compute $\Tr A_h(K_1 - \Id)$. For that we use Proposition~\ref{prop:tr}. Applying \eqref{T-tr-formula} twice, we get
\begin{equation}\begin{gathered}
\Tr A_h K_1 = (2\pi h)^{-(d-1)} \int_{\{ |\eta| \leq R_* \} \setminus
  \mathcal{I}} \rho(|\eta|^2) \, d\omega \, d\eta + o(h^{-(d-1)})\\
\Tr A_h  = (2\pi h)^{-(d-1)} \int_{ \{ |\eta| \leq R_* \}}
\rho(|\eta|^2) \, d\omega \, d\eta + o(h^{-(d-1)})
\end{gathered}\end{equation}
since the fixed point set of $K_1$ is given by $\mathcal{I}^c \cup \mathcal{F}_1$, and by hypothesis, $\mathcal{F}_1$ has measure zero. Subtracting, we find
\begin{equation}
\Tr A_h(K_1 - \Id) = - (2\pi h)^{-(d-1)} \int_{\mathcal{I}}
\rho(|\eta|^2) \, d\omega \, d\eta  = (2\pi h)^{-(d-1)}
\Vol(\mathcal{I}) + o(h^{-(d-1)}),
\end{equation}
This proves the trace formula \eqref{eq:traceformula} for the
polynomial $z -1$.

Finally, consider an arbitrary polynomial $p(z)$ with $p(1) = 0$.  On
the circle $\absv{z} = 1$, since $\overline{z} = z^{-1}$, $p$ can be
written $p(z) = \sum_{- d \le k \le d} a_{k} z^{k}$ for $a_{k} \in
\mathbb{C}$.  If we define $p_{k}(z) = z^{k} - 1$, then $p(z) =
\sum_{0 < \absv{k} \le d} a_{k} p_{k}(z)$.  (Proof: $p(z) -
\sum_{0 < \absv{k} \le d} a_{k} p_{k}(z)$ is constant and vanishes at
$1$.)  By the linearity of the trace, it suffices to prove that
for all $k \in \mathbb{Z}$ with $k \neq 0$,
\begin{equation}
  \label{eq:generalpoly}
  \Tr(S_{h}^{k} - 1) = - \frac{1}{(2\pi h)^{d-1}} \Vol(\mathcal{I}) +
  o(h^{-(d - 1)}).
\end{equation}
This will follow exactly as the case for $k = 1$ if we can show that
 $S_{h}^{k} = K_{1}' + K_{2}' + K_{3}'$ where the $K_{i}'$ have all of
 the same properties as the $K_{i}$ from \eqref{eq:decomposition} with
 the only modification being that $K_{1}'$, still a semiclassical FIO with
microsupport in $\mathcal{I}_{2\epsilon}$, now has canonical relation
\begin{equation*}
  C_{k} := \set{ (\omega, \eta, \omega', -\eta') : \mathcal{S}^{k}(\omega,
    \eta) = (\omega', \eta')}.
\end{equation*}
where $\mathcal{S}$ is the sojourn map \eqref{eq:sojourn}.  All of the relevant properties follow by applying the standard
composition theorems for semiclassical FIOs and for semiclassical
pseudodifferential operators (see e.g.\ \cite{Zw2012}, \cite[Section
8.2.1]{GS2010}) to $S_{h}^{k} =
(K_{1} + K_{2} + K_{3})^{k}$.  Now all the computations in the case of
$z - 1$ work in this case with the sole change that the trace of
$A_h K_{1}'$ is an integral over $\mathcal{F}_{k} \cup \{ |\eta| \leq R_* \} - \mathcal{I})$, but Assumption \ref{thm:dynass}
implies that only the integral over $T^{*}\mathbb{S}^{d-1} -
\mathcal{I}$ contributes, as in the $z - 1$ case.


\end{proof}


\section{Equidistribution}\label{sec:weightedequi}

We now state a strengthened version of Theorem \ref{thm:prethm}.
We define the norm
\begin{equation}
  \label{eq:weightednorm}
  \norm[w]{f} = \sup_{\absv{z} = 1, z \neq 1} \absv{\frac{f(z)}{z -
      1}} \mbox{ for } f \colon \mathbb{S}^{1} \lra \mathbb{C}.
\end{equation}
and let
\begin{equation}\label{eq:Cw}
C_{w}^{0}(\mathbb{S}^{1}) := \set{ f \in C^{0}(\mathbb{S}^{1}) : f /
  (z - 1) \mbox{ is continuous} }.
\end{equation}
Then $C_{w}^{0}(\mathbb{S}^{1})$ with the norm $\norm[w]{\cdot}$
is a Banach space.  (Here $C^{0}(\mathbb{S}^{1})$ is the space of
continuous, complex-valued functions on the unit circle.)  
We will
prove the following.
\begin{theorem}\label{thm:mainthm}
  For any $f \in C^{0}_{w}(\mathbb{S}^{1})$, the integral $\la
  \mu_{h}, f \ra$ defined in \eqref{eq:integralh} is finite for all $h$, and
  \begin{equation}
    \label{eq:equicompactw}
    \lim_{h \to 0} \la \mu_{h}, f \ra = \frac{1}{2\pi} \int_{0}^{2\pi}
    f(e^{i\phi}) d\phi.
  \end{equation}
\end{theorem}

This theorem follows in a straightforward way from the following two lemmas. 

\begin{lemma}\label{lem:uniform-bd} There exists $C > 0$ such that for all sufficiently small $h$, and all $f \in C_{w}^{0}(\mathbb{S}^{1})$, we have 
\begin{equation}
  \label{eq:estimate}
\Big|  \la \mu_{h}, f \ra \Big| \le C \norm[w]{f}.
\end{equation}
\end{lemma}

\begin{lemma}\label{lem:dense} The set 
$\mbox{Poly}_{1}(\mathbb{S}^{1})$ of polynomials $p(z)$ with $p(1) = 0$  are dense in $C_{w}^{0}(\mathbb{S}^{1})$. 
\end{lemma}

\begin{proof}[Proof of Lemma~\ref{lem:uniform-bd}]
Given $\delta > 0$, write
\begin{equation}
  \label{eq:sumemup}
  \begin{split}
    \la \mu_{h}, f \ra &= \frac{1}{c_{V}} h^{d-1}
    \sum_{\mbox{spec}(S_{h})} f(e^{i\beta_{h,n}}) \\
    &=\frac{1}{c_{V}} h^{d-1}\lp 
    \sum_{\absv{\beta_{h,n}} \ge \delta} f(e^{i\beta_{h,n}}) + \sum_{\absv{\beta_{h,n}} < \delta} f(e^{i\beta_{h,n}})\rp,
  \end{split}
\end{equation}
where here and below we choose the branch for which $\beta_{h,n} \in
(-\pi, \pi]$.  To estimate these we will use the following: with $R'$, $c$,  $C$ as in Lemma~\ref{thm:veryclosetozero}, 
 there exists $a > 0$, such that for
every $L > R'/h$,  
\begin{equation}
  \label{eq:howmanybadguys}
  \mbox{ there are at most } a L^{d-1} \mbox{ eigenvalues of
    $S_{h} - 1$ larger than } Ce^{-cL}.
\end{equation}
This follows from \eqref{eq:almosteigenfunction}. To see this, let
$Z(a, L)$ be the subspace of $L^2(\mathbb{S}^{d-1})$ spanned by the
eigenfunctions of $S_h - \Id$ with eigenvalues larger than $C
e^{-cL}$.  Then 
\begin{equation}
\text{$\| (S_h - \Id) u \| \geq C e^{-cL} \| u \|$ for all $u \in Z(a, L)$.}
\label{ppp}\end{equation}
If the dimension of $Z(a, L)$ were greater than $d_1 + \dots + d_L$, which is bounded by $a L^{d-1}$ for suitable $a$, then $Z(a, L)$ would contain a spherical harmonic $\phi_{lm}$ with $l > L$. But putting $u = \phi_{lm}$ in \eqref{ppp} contradicts \eqref{eq:almosteigenfunction}. This establishes \eqref{eq:howmanybadguys}.

To analyze \eqref{eq:sumemup}, first we show that the sum over the eigenvalues $\absv{\beta_{h,n}}
\ge \delta$ is bounded
by $c \norm[w]{f}$.  Each term in the sum satisfies $\absv{f(z)} \le
\norm[w]{f} \absv{z - 1} \le 2 \norm[w]{f}$.  Choosing $L = C/h$ for
large $C$, we have
$ce^{-L/c} \le \delta/2$ for $h$ small, and thus by \eqref{eq:howmanybadguys}
at most $c'/h^{d-1}$ eigenvalues $S_{h} - 1$ larger than $\delta$,
where $c'$ is independent of $\delta$.
Thus
\begin{equation}
  \label{eq:1}
  \frac{1}{c_{V}}h^{d-1} \sum_{\absv{\beta_{h,n}} \ge
    \delta} f(e^{i\beta_{h,n}})  < c' \norm[w]{f}
\end{equation}

Now we estimate the sum over the eigenvalues $\absv{\beta_{h,n}}
< \delta$.  Write
\begin{equation*}
   \sum_{\absv{\beta_{h,n}} < \delta} f(e^{i\beta_{h,n}}) = \sum_{j
   = 0}^{\infty} \sum_{A(\delta,j)} f(e^{i\beta_{h,n}}),
\end{equation*}
where 
\begin{equation*}
  \label{eq:3}
  A(\delta,j) := \set{\beta_{h,n} : \absv{\beta_{h,n}} \in [ \delta 2^{-(j + 1)},
   \delta 2^{-j})}.
\end{equation*}
Taking $L$ such that $c e^{-L / c} = \delta 2^{-(j + 1)}$, there are at
most $c' \lp j\log 2 - \log \delta \rp^{d - 1}$ eigenvalues in
$A(\delta,j)$ for some $c'$ independent of $\delta$.  But in $A(\delta, j)$, $\absv{f(z)} < \norm[w]{f}
\delta 2^{-j}$. Thus, for some $c > 0$ whose value changes from line
to line,
\begin{equation*}
  \begin{split}
    \absv{\sum_{j = 0}^{\infty} \sum_{A(\delta,j)} f(e^{i\beta_{h,n}})
    } &\le \sum_{j = 0}^{\infty} \absv{A(\delta,j)} \norm[w]{f}
    \delta 2^{-j} \\
    &\le \sum_{j = 0}^{\infty} c \lp j\log 2 - \log \delta \rp^{d - 1} \norm[w]{f}
    \delta 2^{-j}  \\
    &\le  c \delta^{1/2} \norm[w]{f}.
  \end{split}
\end{equation*}
Thus
\begin{equation*}
  \la \mu_{h}, f \ra \le  
  c (1 + \delta^{1/2})  \norm[w]{f}. 
\end{equation*}
 This
proves \eqref{eq:estimate}.

\end{proof}

\begin{proof}[Proof of Lemma~\ref{lem:dense}]
note that given $f \in C^{0}_{w}(\mathbb{S}^{1})$, by
definition $f/(z-1)$ is continuous, so by the Stone-Weirstrass theorem, for any
$\epsilon > 0$ there is a polynomial $\wt{p}$ so that
\begin{equation}\label{eq:density}
  \begin{split}
    \epsilon &> \sup_{z \in \mathbb{S}^{1}}\absv{ f/(z-1) - \wt{p}} 
    =  \sup_{z \in \mathbb{S}^{1}}\absv{ \frac{1}{z - 1} \lp f - (z -
      1)\wt{p}\rp} = \norm[w]{f - (z - 1)\wt{p}}.
  \end{split}
\end{equation}
Letting $p = (z-1)\wt{p}$ gives the desired density.
\end{proof}

\begin{proof}[Proof of Theorem \ref{thm:mainthm}]
By Lemma~\ref{lem:uniform-bd}, $\mu_h$ is a bounded linear operator on
$C^{0}_{w}(\mathbb{S}^{1})$, uniformly as $h \to 0$. Therefore to
prove \eqref{eq:equicompactw}, it is only necessary to prove it on a dense subspace. By Lemma~\ref{lem:dense}, therefore, it suffices to prove \eqref{eq:equicompactw} for polynomials that vanish at $z=1$. But this was done in Proposition~\ref{thm:traceone}, so the proof is complete. 
\end{proof}

Finally, we prove Corollary \ref{thm:countingcor}.
\begin{proof}[Proof of Corollary \ref{thm:countingcor}]
Let $1_{[\phi_{0}, \phi_{1}]}$ denote the characteristic function of
the sector on $\mathbb{S}^{1}$ between angles $\phi_{0}$ and
$\phi_{1}$ not passing through $1$, and let $f$ and $g$ be continuous,
positive functions on $\mathbb{S}^{1}$ satisfying $f \le 1_{[\phi_{0}, \phi_{1}]} \le g$.
Then $\Tr f(S_{h}) \le \Tr 1_{[\phi_{0}, \phi_{1}]}(S_{h}) \le \Tr
g(S_{h})$, so 
\begin{equation}
  \label{eq:squeeze}
  \begin{split}
    \frac{1}{2\pi}\oint f(e^{i\phi}) d\phi &= \lim_{h \to 0} \la \mu_{h}, f
    \ra  \\
    &\le \frac{1}{c_{V}} (2\pi h)^{d-1} \Tr 1_{[\phi_{0}, \phi_{1}]}(S_{h}) \\
    &\le \lim_{h \to 0} \la \mu_{h}, g \ra  = \frac{1}{2\pi}\oint
    g(e^{i\phi}) d\phi
  \end{split}
\end{equation}
Letting the integrals of $f$ and $g$ tend to $(\phi_{1} - \phi_{0}) /
(2\pi)$ and using $\Tr 1_{[\phi_{0}, \phi_{1}]}(S_{h})  =
N_{h}(\phi_{0}, \phi_{1})$ proves the corollary.
\end{proof}


\section{Comparison of results with \cite{DGHH2013}}\label{sec:comp}

In \cite{DGHH2013}, Datchev, Humphries and the first two authors proved the following.  For a
Hamiltonian $H$ as in \eqref{eq:operator}, assume that
  the potential function $V$ is central ($V = V(\absv{x})$).  The
  radius of the convex hull of the support, which will be relevant
  below, we denote by
  \begin{equation}
    \label{eq:radiusconvexhull}
    R := \inf\set{R' : \mbox{supp} (V) \subset B_{R'}(0) }.
  \end{equation}
  In this case, the scattering matrix $S_{h}$ is diagonalized by the spherical
  harmonics.  Denoting an arbitrary spherical harmonic by $\phi_{l}$
  where $\Delta_{\mathbb{S}^{d-1}} \phi_{l} = l(l + n - 2)\phi_{l}$,
  as discussed in \cite{DGHH2013}, the eigenvalue of $\phi_{l}$ for
$S_{h}$ depends only on the angular momentum $l$, so we may define
  \begin{equation}
    \label{eq:centraleigenval}
    S_{h} \phi_{l} = e^{i \beta_{h,l}} \phi_{l}, 
  \end{equation}
keeping in mind that the eigenvalue $e^{i \beta_{h,l}}$ has
multiplicity $d_{l}$ from \eqref{eq:multiplicity}.
  \begin{theorem}\label{thm:centralequi}
  For a Hamiltonian $H$ as in \eqref{eq:operator} with $V$
  central and $R$ as in \eqref{eq:radiusconvexhull}.  Let
  $\Sigma(\alpha)$ denote the scattering angle function
  (see \cite[Section 2]{DGHH2013}), and assume that
  \begin{equation}
    \label{eq:scatteringangleassumption}
    \Sigma'(\alpha) \mbox{ has finitely many zeros in } [0, R).
  \end{equation}
Then the set of eigenvalues $e^{i \beta_{h,l}}$ with $l \le R/h$
equidistribute around the unit
circle, meaning that, if we let $\overline{N}_{h}(\phi_{0}, \phi_{1})$ denote the
number of $\beta_{h,l}$ with $l \le R/h$ satisfying $\phi_{0} \le
\beta_{h,l} \le \phi_{1}$ counted with multiplicity $m_{d}(l) = \dim
\ker \Delta_{\mathbb{S}^{d-1}} - l(l + n - 2)$, then
\begin{equation}
  \label{eq:equicentral}
\sup_{0 \le \phi_{0} < \phi_{1} \le 2 \pi} \absv{
  \frac{\overline{N}_{h}(\phi_{0}, \phi_{1})}{\overline{N}_{h}(0, 2\pi)} - \frac{\phi_{1} -
    \phi_{0}}{2\pi}  } \to 0 \mbox{ as } h \to 0.
\end{equation}
\end{theorem}
Note that the difference between $\overline{N}_{h}$ in the theorem and
$N_{h}$ in Theorem \ref{thm:countingcor} is that $\overline{N}_{h}$
deliberately excludes the eigenvalues corresponding to spherical
harmonic with angular momenta $\ge R/h$. The number
$\overline{N}_{h}(0, 2\pi)$ is by definition the total number of
spherical harmonics
with angular momentum $l \le R/h$.  In particular, by \cite[Corollary
4.3]{taylor:vol2}, we
have the leading order expansion
\begin{equation}
  \label{eq:weylasymptotics}
  \overline{N}_{h}(0, 2\pi) = 2\frac{R^{d-1}}{(d - 1)!} h^{-(d - 1)} + O(h^{-(d-2)}).
\end{equation}
For a central potential with scattering angle $\Sigma$ satisfying
\eqref{eq:scatteringangleassumption}, the interaction region satisfies
$\mathcal{I} = \set{(\omega, \eta) : |\eta| \le R }$,
since a bicharacteristic $t\omega + \eta$ hits the potential
if and only if $\absv{\eta} \le R$.  Thus the fixed point
set is actually finite in this case, and the hypotheses of Theorem
\ref{thm:mainthm} hold.    Furthermore,
\begin{equation}
  \label{eq:interractionarea}
  \begin{split}
    \Vol(\mathcal{I}) &= \Vol B^{d-1}(R) \times
    \Vol{\mathbb{S}^{d-1}} 
    = 2\frac{R^{d-1}}{(d - 1)!}  (2\pi)^{d-1}
  \end{split}
\end{equation}
where $B^{d-1}(R)$ is the ball of radius $R$ in $\mathbb{R}^{d-1}$,
$\mathbb{S}^{d-1}$ is the unit sphere, and both quantites on the right
are the Riemannian volumes.   Thus the Theorem \ref{thm:centralequi}
implies Theorem \ref{thm:countingcor} above.

On the other hand, the $l^{th}$ fixed point set $\mathcal{F}_{l}$ (see
\eqref{eq:interactionfixed}) is equal to those $(\omega, \eta)$ with
$\absv{\eta} \le R$ satisfying $\Sigma(\absv{\eta})
\in \frac{1}{l} 2\pi \mathbb{Z}$.  The assumption on $\Sigma$ in
\eqref{eq:scatteringangleassumption} is thus stronger than Assumption
\ref{thm:dynass}.

  As for the other eigenvalues,
i.e.\ the $e^{i \beta_{h,l}}$ with $l > R/h$, in
\cite{DGHH2013} we showed that they
are very close to $1$.
\begin{theorem}\label{thm:centralseparation}
  For $V$ central and $R$ as in \eqref{eq:radiusconvexhull}, let
  $\kappa \in (0, 1)$.  Then the $e^{i\beta_{h,l}}$ with $l \ge (R +
  h^{\kappa})/h$ satisfy
  \begin{equation}
    \label{eq:close}
    \absv{e^{i\beta_{h,l}} - 1} = O(h^{\infty}) \mbox{ as } h \to 0.
  \end{equation}
\end{theorem}
This theorem is weaker than Lemma \ref{thm:veryclosetozero} in the
sense that it does not give exponential decay, and moreover gives no
decay with respect to $l$.  On the other hand, Theorem
\ref{thm:centralseparation} is stronger in the sense that it holds on
a larger region than Lemma \ref{thm:veryclosetozero}.  This is true not
only because Theorem \ref{thm:centralseparation} holds on a region
asymptotically approaching the interaction region, but because the
Lemma \ref{thm:veryclosetozero} is essentially control over $S_{h} -
\Id$ on a region of phase space outside of a \textit{ball bundle} containing
the interaction region.  But in general there are portions of the
complement of the interaction region, $\mathcal{I}^{c}$, that do not lie in
the complement of a ball bundle containing $\mathcal{I}$.  The work in
this paper indicates that, for given $h$, there are approximately $c_{V} (2\pi
h)^{-(d-1)}$ eigenvalues that are equidistributed around the unit
circle, and the rest are close to $1$.  Theorem
\ref{thm:centralseparation} quantifies the latter statement by separating the
spectrum into two parts, one equidistributing and the other close to
$1$, while Lemma \ref{thm:veryclosetozero} does not.  It would be
interesting to know if such a separation of the spectrum were possible
in the non-central setting.


\section{Appendix: Trace formula for semiclassical FIOs}

In this section we shall prove a trace formula for semiclassical FIOs on a compact manifold. In fact, we prefer here to think in terms of Legendre distributions, for reasons that we now explain. 

Let $M$ be a compact manifold of dimension $d$. Then $T^* M \times
\RR$ is a contact manifold, with canonical contact form $\alpha :=
\zeta \cdot dz - d\tau$, where $z$ are local coordinates on $M$,
$\zeta$ dual coordinates on the fibres of $T^* M$, and $\tau$ is the
coordinate on $\RR$. A Legendre submanifold, $L$, of $T^* M \times \RR$ is a
submanifold of dimension $d$ on which $\alpha$ vanishes
identically. There is a very close relationship between Legendre
submanifolds on $T^* M \times \RR$ and Lagrangian submanifolds of $T^*
M$. To describe this, let $\pi : T^* M \times \RR \to T^* M$ denote
the canonical projection. Then the vector field $\partial_\tau$ is
never tangent to $L$ due to the vanishing of $\alpha$ on $L$, so $\pi$
is locally a diffeomorphism from $L$ to $\Lambda = \pi(L)$. Working
locally, i.e.~ restricting attention to a small open set of $L$, we
can assume $\Lambda$ is a submanifold. Moreover, since $d\alpha =
d\zeta \wedge dz$ vanishes on $L$, it also vanishes on $\Lambda$,
i.e. ~ $\Lambda$ is Lagrangian.

Conversely, suppose that $\Lambda \subset T^* M$ is a Lagrangian submanifold. Then since $d(\zeta \cdot dz) = 0$ on $\Lambda$, it is (at least locally) the differential of a smooth function, say $f$, on $\Lambda$. Then it is easy to check that
$$
L = \{ (z, \zeta, \tau) \mid (z, \zeta) \in \Lambda, \ \tau = f(z, \zeta) \} \subset T^* M \times \RR
$$
is a Legendre submanifold of $T^* M \times \RR$. Notice that $f$ is determined up to an additive constant, and hence $L$ is determined up to shifting $\tau$ by an additive constant. 

Now consider a semiclassical Lagrangian distribution $A$ associated to $\Lambda$. This will have an oscillatory integral representation in terms of (one or several) phase function(s) $\Phi$, depending on $z$ and an auxiliary coordinate $v \in \RR^k$,  locally parametrizing $\Lambda$, in the sense that locally, we have
$$
\Lambda = \{ (z, d_z\Phi(z, v)) \mid d_v\Phi(z, v) = 0 \}.
$$
Notice that $\Phi$ is undetermined up to an additive constant. Adding $c$ to $\Phi$ will have the effect of changing the Lagrangian distribution by $e^{ic/h}$. In some settings this is irrelevant: for example, if $A$ represents a family of eigenfunctions or quasimodes, multiplication by a complex unit is harmless. However, in other cases, as in the present paper where we consider the semiclassical scattering matrix, the Schwartz kernel is determined uniquely and multiplication by complex units is not harmless (especially when we study the spectrum of the scattering matrix!). 

We can now explain why we consider it preferable to describe the semiclassical scattering matrix as a Legendre distribution. It is because the $\tau$ coordinate in $T^* M \times \RR$ is given precisely by the value of the phase function. Thus parametrizing $L$ means finding a function $\Phi(z, v)$ such that 
$$
L = \{ (z, d_z\Phi(z, v), \tau = \Phi(z, v)) \mid d_v\Phi(z, v) = 0 \}.
$$
Thus the $\tau$ coordinate eliminates the indeterminacy of $\Phi$ up to an additive constant. Another way of putting this is that multiplication of $A$ by $e^{ic/h}$ would give a family of Lagrangian distributions associated to the same Lagrangian submanifold $\Lambda$, but they would all be associated to different Legendre submanifolds $L_c$. 

The $\tau$ coordinate plays a role in the following trace formula for semiclassical FIOs. (Notice that this does not happen for the trace formula for homogeneous FIOs, since the value of the phase function on the Lagrangian is always zero in the homogeneous case.)

\begin{proposition}\label{prop:tr}
Let $T$ be a semiclassical FIO of (semiclassical) order zero and compact microsupport acting on half-densities on a compact manifold $M$ of dimension $n$, associated to a Legendre submanifold $L \subset T^* M^2 \times \RR$ that maps to the graph of a 
canonical transformation $\sigma$ under the projection $T^* M^2 \times \RR \to T^* M^2$. Then $T$ is trace class and satisfies
\begin{equation}
\Tr T = (2\pi h)^{-n} \int_{\mathrm{Fix}(\sigma)} e^{i\tau(x, \xi)/h} \sigma(T)(x, \xi) |dx d\xi|^{1/2} + o(h^{-n}). 
\label{T-tr-formula}\end{equation}
\end{proposition}

\begin{proof}
Since $T$ has compact microsupport, it is trace class and the trace is given by the integral of the Schwartz kernel restricted to the diagonal. Also, due to the assumption of compact microsupport, $T$ can be written as a finite number of oscillatory integrals of the form 
\begin{equation}
(2\pi h)^{-n} \int_{\RR^n} e^{i\Phi(x, y, v)/h} a(x, y, v, h) \, dv |dx dy|^{1/2}, 
\label{piece}\end{equation}
where $\Phi$ locally parametrizes $L$ nondegenerately in the neighbourhood of a point $q = (x, \xi, y, \eta, \tau) \in L$ and $a$ is smooth and compactly supported in $v$, in fact supported in an arbitrarily small neighbourhood of the point $(q', 0)$ such that $$q' = (x_0, y_0, v_0), \ d_v\Phi(q') = 0 \text{ and }(x, d_x\Phi(q'), y, d_y \Phi(q'), \Phi) = q.$$

\textbf{Case 1.} Suppose that $q$ is not in the set 
$$
N^* \Delta \times \RR = \{ (x, \xi, x, -\xi, \tau) \}. 
$$
Then, either (i) $d_x \Phi(q') + d_y \Phi(q') \neq 0$ or (ii) $x_0 \neq y_0$. In the latter case, (ii), 
by restricting the support of $a$ close to $(q', 0)$ the restriction of the Schwartz kernel of $T$ to the diagonal vanishes identically for small $h$, trivially implying \eqref{T-tr-formula}. If $x_0 = y_0$ but (i) holds, then the trace is given by 
\begin{equation}
(2\pi h)^{-n} \int_{\RR^n} e^{i\Phi(x, x, v)/h} a(x, x, v, h) \, dv \, dx 
\label{Ttrace}\end{equation}
and (i) implies that $d_x (\Phi(x_0, x_0, v_0)) \neq 0$. Let us write $\tPhi(x, v) = \Phi(x,x,v)$.  By taking the support of $a$ sufficiently close to $(q', 0)$, we have $d_x \tPhi(x, v) \neq 0$ whenever $(x, x, v, h)$ is in the support of $a$. Then using the identity
$$
e^{i\tPhi(x, v)/h} = \Big( \frac{h}{i d_x (\tPhi(x, v))} d_x \Big)^N e^{i\tPhi(x, v)/h}
$$
and integrating by parts, we see that \eqref{Ttrace} is $O(h^\infty)$ in this case. 

\textbf{Case 2.} Suppose that $q \in N^* \Delta \times \RR$, but that
the tangent map $D\sigma$ at $\pi(q)$ is not the identity. In this
case, the measure of $\Fix(\sigma)$ is zero, and we will show that the trace is
$O(h^{1/2 - n})$, thereby obtaining Proposition~\ref{prop:tr} in this case.  We
claim that the phase function $\tPhi(x,v)$ restricted to
$x=y$ is nondegenerate in at least one direction, i.e. there is at
least one nonzero component in $d^2 \tPhi(x_0, v_0)$ is
nonzero. Applying the stationary phase lemma in one non-degenerate direction, we gain
$h^{1/2}$, i.e. we find that the trace is $O(h^{1/2 - n})$.

To show this, recall that nondegeneracy of $\Phi$ means, by definition, that the differentials $d(d_{v_i} \Phi)$ are linearly independent whenever $d_v \Phi = 0$. This implies in particular that the submanifold $C$ in $(x, y, v)$-space given by $d_v \Phi = 0$ is a smooth submanifold of codimension $n = \dim v$, and that the map
\begin{equation}
C \ni (x, y, v) \mapsto (x, d_x \Phi, y, d_y \Phi) \in T^* M^2
\label{diffeo}\end{equation}
is a diffeomorphism from $C$ to the graph of $\sigma$. A tangent vector to the graph of $\sigma$ can be represented in $(x, y, v)$-space as a vector $X = a \cdot \partial_x + b \cdot \partial_y + c \cdot \partial_v$ such that $X(d_v \Phi) = 0$. 
Let us assume that 
\begin{equation}d^2_{vv} \tPhi = 0 \text{ and } d^2_{xv} \tPhi = 0 \text{ at }(x_0, v_0), 
\label{assumetPhi}\end{equation}
otherwise there is nothing to prove using the remarks in the first paragraph of Case 2. This implies that 
\begin{equation}
d^2_{vv} \Phi = 0 \text{ and } d^2_{xv} + d^2_{yv}\Phi = 0 \text{ at } (x_0, x_0, v_0). 
\label{assumePhi}\end{equation}
Combining this with the linear independence of $d (d_{v_i} \Phi)$ implies that $d^2_{xv} \Phi$ is nonsingular at $(x_0, x_0, v_0)$, since otherwise the matrix $d_{x,y,v} d_v\Phi$ would have rank strictly less than $n$. Then \eqref{assumePhi}, the nondegeneracy of $d^2_{xv} \Phi$, and $X(d_v \Phi) =  0$ at $(x_0, x_0, v_0)$ imply that 
$$
(a - b) d^2_{xv} \Phi = 0 \implies a = b.
$$
So vectors tangent to $C$ at $(x_0, x_0, v_0)$ are represented by vectors $X = (a, a, c)$ with $a$ and $c$ arbitrary. 

Given $X$ tangent to $C$, let $Y$ be the vector in $T^* M^2$ that
corresponds to it under the diffeomorphism \eqref{diffeo}. By
assumption in Case 2, there exists $X$ such that, writing $Y  = (Y_1,
Y_2)$ with $Y_{1}$ the left and $Y_{2}$ the right components in
$T^{*}M \times T^{*}M$, does \textit{not} satisfy $Y_1 = (\alpha, \beta)$ and $Y_2 = (\alpha,
-\beta)$. Here $Y_i$ are given by 
\begin{equation}\begin{gathered}
Y_1 = a_i  \partial_{x_i} + \Big( a_i  d^2_{x_ix_j} \Phi + b_i d^2_{y_ix_j}\Phi + c_i d^2_{v_ix_j} \Phi \Big) \partial_{\xi_j} , \\
Y_2 = b_i  \partial_{y_i} + \Big( a_i  d^2_{x_iy_j} \Phi + b_i d^2_{y_iy_j}\Phi + c_i d^2_{v_iy_j} \Phi \Big) \partial_{\eta_j} 
\end{gathered}\end{equation}
and we sum over repeated indices. Using $a=b$ we see that the $\partial_x$ components of $Y_1$ and $Y_2$ are equal. Therefore   the $\partial_\xi$ components of $Y_1$ and $Y_2$ cannot sum to zero.  Using \eqref{assumePhi}, this yields the condition 
$$
a_i \Big(  d^2_{x_ix_j} \Phi + d^2_{y_ix_j}\Phi + d^2_{x_iy_j} \Phi + d^2_{y_iy_j}\Phi \Big) \neq 0 \text{ at } (x_0, x_0, v_0),
$$
which implies that the matrix $d^2_{xx} \tPhi(x_0, v_0) \neq 0$. This proves the nonvanishing of the Hessian of $\tPhi$ and completes the proof in Case 2.


\textbf{Case 3.} The remaining case is that $q$ is in $N^* \Delta \times \RR$, and every neighbourhood of $q$ intersects $N^* \Delta \times \RR$ in a set of positive measure. This is only possible if $D\sigma$, the derivative of the map $\sigma$, is the identity at $q$. Now choose a set of local coordinates $x$ on the left factor of $M$ and denote by $x'$ the same coordinates on the right factor of $M$. Similarly denote the dual coordinates $\xi, \xi'$. Then $(x', \xi)$ form coordinates locally on $L$ near $q$. This follows from the fact that $(x, \xi)$ are coordinates on $L$ due to the fact that $\sigma$ is a canonical transformation, and the fact that $D \sigma$ is the identity, implying that $\partial x'/\partial x$ (keeping $\xi$ fixed) is the identity at $q$. It follows that we can choose the phase function $\Phi$ of the form $\Phi(x, x', \xi)$ where $d_x \Phi = \xi$. 

With a phase function chosen as above, consider the integral \eqref{Ttrace}. We introduce a cutoff function $\chi(x, \xi)$ as follows.  Given $\epsilon > 0$, we find an open set $U$ containing $L \cap N^* \Delta \times \RR$ with measure difference $| U \setminus (L \cap N^* \Delta \times \RR)|$ less than $\epsilon$. We can also find  small enlargements $U_\delta$, $U_{2\delta}$ of $U$ such that the measure difference satisfies $|U_\delta \setminus U| < \epsilon$ and $|U_{2\delta} \setminus U_\delta| < \epsilon$.   Let $\chi(x, \xi)$ (interpreted as a function on $L$) be equal to $1$ on $U_\delta$, and supported in $U_{2\delta}$. With $1 - \chi$ inserted in \eqref{Ttrace}, the integral is $O( C(\epsilon) h^{1-n})$ using the argument of case 1, since the differential of $\tPhi$ does not vanish on the support of $1 - \chi$. With $\chi$ inserted, we obtain the integral  
$$
(2\pi h)^{-n} \int_{\RR^n} e^{i\tPhi(x, \xi)/h} a(x, x, \xi, h) \chi(x, \xi) \, dx  \, d\xi 
$$
As $\epsilon \to 0$, the function $\chi$ tends to the characteristic function of $\Fix(\sigma)$ almost everywhere. Therefore, using the dominated convergence theorem, we find that 
$$
\Tr T_h = (2\pi h)^{-n} \Big( \int_{\Fix(\sigma)} e^{i\Phi(x, x, \xi)/h} a(x, x, \xi, h) \,  dx \, d\xi  + o(1) \Big)  + O(C(\epsilon) h^{1-n}).
$$
Here the $o(1)$ is as $\epsilon \to 0$, coming from replacing $\chi$ with its limit, the characteristic function of $\Fix(\sigma)$. Choosing $h$ small enough, we have 
$O(C(\epsilon) h^{1-n}) = o(h^{-n})$. Finally we note that in these coordinates, the symbol of $T$ is $a(x, \xi) |dx d\xi|^{1/2}$ (and the Maslov factors are trivial since $T$ is associated to a canonical transformation $\sigma$ with $D \sigma$ equal to the identity at $q$). This completes the proof of Proposition~\ref{prop:tr}. 
\end{proof}

\bibliographystyle{abbrv}

\end{document}